\documentclass{amsart}

\usepackage{a4wide,amssymb,amsmath,amsthm,color}

\usepackage[bookmarksopen=true,bookmarksnumbered=true, bookmarksopenlevel=2, colorlinks=true,linkcolor=mblue,
citecolor=mblue,urlcolor=mblue, pdfstartview=FitH]{hyperref}
\usepackage[dvipsnames]{xcolor}
\usepackage{pgfplots}
\usepackage{enumerate}
\usepackage{graphicx}
\usepackage{caption}
\usepackage{subcaption}
\graphicspath{ {./images/} }

\definecolor{mblue}{rgb}{0,0,.8}

\newcommand{\N}{\mathbb N}
\newcommand{\Z}{\mathbb Z}
\newcommand{\Q}{\mathbb Q}

\newcommand{\R}{\mathbb R}
\newcommand{\C}{\mathbb C}

\newcommand{\Zmodpm}[1]{\Z/{#1}\Z}

\renewcommand{\a}{\alpha}
\renewcommand{\b}{\beta}
\renewcommand{\d}{\delta}

\renewcommand{\k}{\kappa}

\newcommand{\B}{\mathcal{B}}

\newcommand\Es[0]{E^\ast}

\newtheorem{thm}{Theorem}
\newtheorem{con}{Conjecture}
\newtheorem{lem}[thm]{Lemma}
\newtheorem{remark}[thm]{Remark}

\newtheorem{prop}[thm]{Proposition}
\newtheorem{cor}[thm]{Corollary}
\newtheorem*{con-non}{Conjecture}

\newtheorem{alg}{Algorithm}

 \DeclareMathOperator{\SL}{SL}

\begin{document}

\title[Computations on Overconvergence Rates Related to the Eisenstein Family]{Computations on Overconvergence Rates Related to the Eisenstein Family}
\author{Bryan \ Advocaat}
\address[Bryan \ Advocaat]{Department of Mathematics, University of Luxembourg, Maison du nombre, 6~avenue de la Fonte, L-4364 Esch-sur-Alzette, Luxembourg, and Department of Mathematical Sciences, University of Copenhagen, Universitetsparken 5, DK-2100 Copenhagen \O ,
Denmark.}
\email{bryan.advocaat@uni.lu}

\bibliographystyle{plain}

\begin{abstract} We provide for primes $p\ge 5$ a method to compute valuations appearing in the "formal" Katz expansion of the family $\frac{\Es_\k}{V(\Es_\k)}$ derived from the family of Eisenstein series $\Es_\k$. The overconvergence rates of the members of this family go back to a conjecture from Coleman. We will describe two algorithms: the first one to compute the Katz expansion of an overconvergent modular form and the second one, which uses the first algorithm, to compute valuations appearing in the "formal" Katz expansion. Based on data obtained using these algorithms we make a precise conjecture about a constant appearing in the overconvergence rates related to the classical Eisenstein series at level $p$. 
\end{abstract}

\maketitle

\section*{Introduction}
In this paper we will obtain computational data regarding the overconvergence rate related to the Eisenstein Family. We will formulate two algorithms; the first one is to compute the Katz expansion (see below for its definition) of an overconvergent modular form, and the second algorithm is able to compute valuations appearing in the "formal Katz expansion" of a family related to the classical Eisenstein series of level $p$. Based on data obtained from the second algorithm, we formulate a conjecture. To state our conjecture precisely, we will start by introducing the necessary terminology. Throughout, $p$ will denote a prime $\geq 5$, and $\nu_{p}$ (or simply $\nu$ if there is no confusion about the prime) will denote the $p$-adic valuation of $\C_{p}$, normalized such that $\nu_{p}(p) = 1$.  We let $\mathcal{W}$, called the weight space, be the group $\text{Hom}_{\text{cont}}(\Z_{p}^{\times},\C_{p}^{\times} )$, i.e. the continuous characters of $\Z_{p}^{\times}$ with values in $\C_{p}$. We denote by $\mathcal{B}$ the subspace of these characters which are trivial when restricted to the $(p-1)$st roots of unity. If we denote by $\mathcal{D}$ the open disk of radius $1$ around the origin in $\C_{p}$ then we can identify $\mathcal{W}$ with $\mathcal{D}$ by sending an element $\kappa \in \mathcal{W}$ to the element $w_{\kappa} := \kappa(p+1)-1 \in \mathcal{D}$. A positive integer $k$ corresponds to the weight given by the character $x \mapsto x^{k}$, and it is precisely a weight in $\mathcal{B}$ if $k$ is divisible by $p-1$. We will denote classical weights just by an integer $k$ instead of by their corresponding character. \\
Then, for a weight $\kappa \in \mathcal{B} \backslash \{ 1 \}$ we have a family interpolating the classical Eisenstein series whose $q$-expansions are given by
\[
\Es_{\kappa}(q) = 1 + \frac{2}{\zeta^{\ast}(\kappa)} \sum_{n=1}^{\infty} \left( \sum_{\substack{ d\mid n \\ p\nmid d}} \kappa(d) d^{-1} \right) \cdot q^n,
\]
where $\zeta^{\ast}$ is the $p$-adic zeta function on $\mathcal{W}$. It has been observed (see for example \cite[Section 5]{BuzCal}) that there exists a power series, $E/VE \in \Z_{p}[[q,w]] $ such that if we are given a weight $\kappa \in \mathcal{B}\backslash \{ 1 \}$, then we have that $(E/VE)(w_\kappa) = \Es_{\kappa}/V(\Es_{\kappa})$, where $V$ is the $p$-adic Frobenius operator, acting on $q$-expansions as $q \mapsto q^{p}$. \\

Our main interest will then be to deduce information about the overconvergence rate of $\Es_{\kappa}/V(\Es_{\kappa})$. To describe the overcovergence rate we use the notion of a Katz expansion. We will give a short description of it, for more details see Section \ref{theory}. Katz showed that for each $i \in \N_{\geq 0}$ there is a splitting 
\begin{equation}\label{splitting}
M_{i(p-1)} (\Z_{p}) = E_{p-1} M_{(i-1)(p-1)}(\Z_{p}) \oplus B_{i}(\Z_{p}),
\end{equation}
where $E_{p-1}$ is the Eisenstein series of weight $p-1$ and level $1$, normalized such that its constant coefficient is $1$ (see \cite{K}). Such a splitting is not unique, but once it has been chosen, Katz has shown that an overconvergent modular form of weight $0$ can be written uniquely as $f = \sum_{i=0}^{\infty} \frac{b_{i}}{E_{p-1}^{i}}$, where $b_{i} \in B_{i}(\Z_{p})$, which is called its Katz expansion, and the values $\nu_{p}(b_{i}) := \text{inf}_{n}(\nu_{p}(a_{n}(b_{i})))$ can be used to measure the overconvergence rate of $f$. From now on, we will fix a splitting described by Lauder (see \cite{L}), which is particularly easy to compute with. For the explicit description of this see Section \ref{theory}. It is shown in \cite{OP} that there exist modular forms $b_{i,j} \in \Z_{p}[[q]]$ for all $i,j \geq 0$  such that if we define $\beta_{i} (w) := \sum_{j=0}^{\infty} b_{ij}w^{j}$ and if $\kappa \in \mathcal{B} \backslash \{ 1 \}$, then the Katz expansion of $\Es_{\kappa} / V(\Es_{\kappa})$ is given by $\sum_{i=0}^{\infty} \frac{\b_i(w(\k))}{E_{p-1}^{i}}$.  In \cite[Section 3.3]{OP}, it is proven that $\nu_{p}(b_{ij}) \geq c_{p}i-j$, where the $c_{p}$ is an explicit constant depending only on $p$. This can be used to give explicit overconvergence rates for $\Es_{\kappa}/V(E_{\kappa})$ for weights $\kappa \in \mathcal{B} \backslash \{ 1 \}$. This $c_{p}$, however, does not seem to be optimal, in the following sense. Denote by $\delta_{p}$ the following quantity
\[
\delta_{p} := \text{inf} \left \{ \frac{ \nu_{p}(b_{ij}) + j}{i} \middle| i \in \Z_{>0}, j \in \Z_{\geq 0} \right \}.
\]
So in particular, it is known by \cite[Theorem B]{OP}, that $\delta_{p} \geq c_{p}$. The main purpose of this paper will then be to compute approximations of the constant $\delta_{p}$ for different primes $p$. In particular, we conjecture the following.
\begin{con}\label{conjecture}
Let the $b_{i,j}$ be as above, then we have that $\nu(b_{i,j}) \geq d_{p}i-j$, for all $i,j \geq 0$, where
\[
d_{p} = \frac{p-1}{p(p+1)}.
\]
\end{con}
Hence we conjecture that $\delta_{p} \geq \frac{p-1}{p(p+1)}$. Note that we do not conjecture that equality holds, but computations for low primes do give explicit values $i$ and $j$ for which we find $\frac{\nu_{p}( b_{ij}) + j }{ i} = d_{p}$, and hence in this case our conjecture would imply $d_{p} = \delta_{p}$. Assuming the conjecture, we have the following corollary.

\begin{cor}\label{cor}
Assume that Conjecture \ref{conjecture} holds. Let $\k \in \B \backslash \{ 1\}$ be a character and let $\mathcal{O}$ be the ring of integers in the extension of $\Q_p$ generated by the values of $\k$.

Then
$$
\frac{\Es_{\k}}{V(\Es_{\k})} \in M_0(\mathcal{O},\ge d_p \cdot \min \{ 1, v_p(w(\k)) \} ) .
$$
\end{cor}
See Section \ref{theory} for the precise definition of $M_0(\mathcal{O},\ge d_p \cdot \min \{ 1, v_p(w(\k)) \} )$. The proof that the conjecture implies Corollary \ref{cor} can be found in \cite[Proof 3.4]{OP}.
The motivation for considering the value $\delta_{p}$ is because of a conjecture Coleman made regarding the overconvergence rate of $ E / V(E)$. His conjecture seems to be too optimistic and in \cite{OP} a counterexample is given, and a slightly different overconvergence rate from Coleman is proven. However, this overconvergence rate seems to be not 'optimal', in the sense that the constant $c_{p}$ is strictly smaller than the conjectured value $\d_{p}$. For $p=2,3$, information about the precise overconvergence rates of $E/VE$ is used to obtain information about the geometry of the eigencurve near the boundary. In particular, it is shown that, close enough to the boundary, the eigencurve is a countable disjoint union of annuli. This provides information regarding the slopes close enough to the boundary, see \cite{BuzCal} and \cite{Roe}. \\

The precise value of $d_{p}$ in Conjecture \ref{conjecture} is based on data obtained using a method, which we will describe in this paper. We will start by giving the necessary theoretical background regarding (formal) Katz expansions. After this, we will provide two algorithms. The first algorithm will take as input a prime $p \geq 5$, two positive integers $n$ and $C$, and a power series in $\Z_{p}[[q]]/(q^{N},p^{C})$ (where $N$ is an explicit constant depending on $n$), and will output the first $n+1$ terms of the Katz expansion, with respect to an explicit splitting of Equation \ref{splitting}, which is particularly useful for computational methods. The second algorithm is the main algorithm and it uses Algorithm 1. It takes as input a prime $p \geq 5$, a nonnegative integer $r$ and a list of integral weights $L = [\kappa_{1}, \ldots, \kappa_{\lambda}]$, for some integer $\lambda \geq 0$. The output will be the values $\nu_{p}(b_{r,j})$ for $0 \leq j \leq r$ if these can be determined exactly. Note that the output allows us to tell if they are indeed exact, or if we cannot conclude the value of some $\nu_{p}(b_{r,j})$. Increasing the number of weights in the input allows us (modulo some technicalities, see the discussion after Algorithm \ref{alg2}) in general to get a conclusive value for $\nu_{p}(b_{r,j})$ for fixed $0 \leq j \leq r$. \\
In the final section we provide data obtained using the second algorithm. In particular, our data shows that for our obtained value we have the bound $\nu_{p}(b_{i,j}) \geq d_{p}i -j$ for  $d_{p} = \frac{p-1}{p(p+1)} $. We indeed know from \cite{OP} that a lower bound of this form exists, but the constant $d_{p}$ differs from the proven constant. \\

Note that for the primes $p$ such that $X(p)$ has genus $0$ (i.e. $p \in \{2,3,5,7,13\}$) we have the so called hauptmodul, defined by
\[
f_{p}(z) := \left( \frac{\Delta(pz)}{\Delta(z)} \right) ^{\frac{1}{p-1}},
\]
where $\Delta$ is the normalized cuspform of level $1$ and weight $12$. This function will generate the function field of $X_{0}(p)$ and can be used to measure the overconvergence rates of overconvergent modular forms (see \cite{Lo}). A result of Buzzard and Calegari \cite{BuzCal} for $p=2$, and for $p=3$ by Roe \cite{Roe}, shows that we can write $E/VE$ as a power series in $\Z_{p}[[w,f]]$, and if we write $E/VE = \sum_{i,j \geq 0} a_{i,j}f^{i}w^{j}$ then we have a lower bound for $\nu_{p}(a_{i,j})$ which in both their cases is linear in $i-j$. However, for other primes $X_{0}(p)$ will be of genus strictly higher than $0$ and hence there will not be a single hauptmodul which can measure the overconvergence rate, but one can replace the theory by using Katz expansions, motivating our interest in information regarding the "Formal Katz expansion".\\

\subsection*{Acknowledgments}
This work was supported by the Luxembourg National Research Fund PRIDE17/1224660/GPS.

\section{Theoretical Background}\label{theory}
We will start by exposing the theory of overconvergent modular forms à la Katz. While nowadays there exists a more geometric and intrinsic definition of an overconvergent modular form by the works of Pilloni \cite{Pil} and Andreatta, Iovita, Stevens \cite{AIS}, the theory by Katz has the advantage that it is very explicit and allows one to do explicit computations with them. Because of this, we will only focus on the definition provided by Katz. We fix a prime $p \geq 5$ and for each integer $k$ we let $M_{k}(\Z_{p})$ denote the space of weight $k$ modular forms of level $1$ with coefficients in $\Z_{p}$ : $(M_{k}(1) \cap \Z[[q]]) \otimes \Z_{p}$. We let $E_{p-1}$ be the classical Eisenstein series of weight $p-1$, normalized such that it has constant coefficient $1$. From Lemma 2.6.1 in \cite{K}, we know that there is a (non-canonical) splitting
\[
M_{k+i(p-1)}(\Z_{p}) = E_{p-1} \cdot M_{k+(i-1)(p-1)}(\Z_{p}) \oplus B_{i}(\Z_{p}),
\]
where the $B_{i}(\Z_{p})$ are free $\Z_{p}$-submodules of $M_{k+i(p-1)}(\Z_{p})$. Katz then shows that, given such a splitting, for $\rho \in \R$, the $\rho$-overconvergent modular forms of weight $0$ and tame level $1$ can be written as a series of the form
\[
\sum_{i=0}^{\infty} \frac{b_{i}}{E_{p-1}^{i}},
\]
where $b_{i} \in B_{i}(\Z_{p})$, $v_{p}(b_{i}) \geq i\rho$ and  $v_{p}(b_{i}) - i\rho \rightarrow \infty$ for $i \rightarrow \infty$. Sometimes, these are called $r$-overconvergent modular forms, where $r\in \C_{p}$ such that $\nu(r)=\rho$. We denote the space of the $\rho$-overconvergent modular forms by $M_{0}(\Z_{p}, \rho)$. If we have an overconvergent modular form, $f \in M_{0}(\Z_{p}, \geq \rho)$, then it can be written as such a series, and we refer to it as the Katz-expansion (even though such a series, of course, depends on the chosen splitting and hence the Katz-expansion is only unique after fixing a splitting).
 We shall also use the following notation. If we have a rational $\rho \in [0,1]$, we let $M_{0}(\Z_{p}, \geq \rho)$ be the $\Z_{p}$-module of forms $f$ such that$ f \in M_{0}(\Z_{p}, \rho')$, for some $\rho' \in \R$, and $v_{p}(b_{i}) \geq i\rho$, for all the coefficients $b_{i}$ of the Katz expansion of $f$. Note that the whole discussion above carries through if we use the ring of integers $\mathcal{O}$ of some finite extension $K/\Q_{p}$, and we can define in a similar fashion the modules $M_{0}(\mathcal{O}, \rho)$ and $M_{0}(\mathcal{O}, \geq \rho)$. \\
 \\

Lauder has given an explicit splitting (see \cite{L}), which is easy to compute with, and we will give a short description of this.  While it is possible to work with higher levels, we will not use this. As we work mainly over $\Z_{p}$, we will suppress this from our notation and denote these spaces by $B_{i}$; if we need to work over another ring we will write it explicitly. To describe the spaces $B_{i}$, we start by defining some auxiliary functions. For $n$ a non-negative integer we define
\[
d_{n} := \left \lfloor \frac{n}{12}  \right \rfloor +  \begin{cases} 
      1 & n \not \equiv 2 \mod{12} \\
      0  & n \equiv 2 \mod{12}, \\

   \end{cases}
\]
that is, $d_{n}$ is the dimension of the classical space of modular forms of weight $n$ and level $1$. We also have the following function

\[
\epsilon(k) := \begin{cases}
	0 & k \equiv 0 \mod{4}\\
	1 &	k \equiv 2 \mod{4}.\\
	\end{cases}
\]
Then, for a fixed $i \geq 0$ and $ j \geq 0$ we define

\begin{equation}\label{gijdef}
g_{i,j} := \Delta^{j}E_{4}^{a}E_{6}^{\epsilon(i(p-1))},
\end{equation}

where 
\[
a = \frac{i(p-1)-12j-6\epsilon(i(p-1))}{4},
\]
and $\Delta$ is the normalized weight $12$ cusp form. Note that for $j=0,\ldots, d_{n}-1$ the numbers $a$ are nonnegative integers and the $g_{i,j}$ are weight $i(p-1)$ modular forms (of level $1$) and the $q$-expansion of $g_{i,j}$ starts with $q^{j}$. Then we put $B_{0}(\Z_{p}) := \Z_{p}$ and for $i > 0$ we let $B_{i}(\Z_{p})$ be the free $Z_{p}$-module spanned by 
\[
\mathcal{B}_{i} := \{ g_{i,j} | d_{(i-1)(p-1)} \leq j \leq d_{i(p-1)} - 1 \}.
\]
The spaces $B_{i}$ then give a splitting as in \eqref{splitting}. Note that if we fix a $j\geq 0$, then there is a unique $i \geq 0$ such that $g_{i,j} \in \mathcal{B}_{i}$; we can find it by picking the unique $i$ such that $d_{(i-1)(p-1)} \leq j \leq d_{i(p-1)} - 1$, we will denote this element by $i_{j}$. If we write $g_{j}$ we mean the element $g_{i_{j},j}$. Note that for any $j \geq 0$ there exists a $g_{j}$, but, depending on the prime, the $\mathcal{B}_{i}$ might be empty for certain $i$. For example, if $p=5$ then we have that $\mathcal{B}_{i} = \emptyset$ unless $i$ is a multiple of $3$.\\

The main use of this specific splitting is that the (infinite) matrix whose rows contain the coefficients of the $q$-expansions of 
\[
g_{0,0}, \ldots, g_{i,d_{(i-1)(p-1)}}, \ldots, g_{i,d_{i(p-1)-1}},
\]
is upper triangular with $1$'s on the diagonal. As the $q$-expansion of $E_{p-1}$ is in $1 + p\Z_{p}[[q]]$, also the $q$-expansion of $E_{p-1}^{-i}$ will be in $1 + p\Z_{p}[[q]]$, and thus the (infinite) matrix whose rows contain the coefficients of the $q$-expansions of 
\[
g_{0,0}, \ldots, g_{i,d_{(i-1)(p-1)}}E_{p-1}^{-i}, \ldots, g_{i,d_{i(p-1)-1}}E_{p-1}^{-i}, \ldots
\]
will also be upper triangular, with $1$'s on the diagonal. This implies that we have an isomorphism $\phi: \prod_{i\ge 0} B_i(\Z_p) \to \Z_p[[q]]$ of $\Z_p$-modules given by
\[
\phi((b_i)_{i\ge 0}) := \sum_{i\ge 0} b_i(q) E_{p-1}^{-i}(q) .
\]

In particular, if we are given the $q$-expansion, say up to $a_{N}(f)q^{N}$, of an overconvergent modular form $f$ this can be turned into an algorithm to compute its Katz expansion (up till some precision), which is the inverse of the map $\phi$.
\\
Our first goal will be to compute the valuations appearing in the "formal Katz expansion" of a family of overconvergent modular forms related to the Eisenstein series. We consider weights, i.e. characters $\kappa: \Z_{p}^{\times} \to \C_{p}^{\times} $. As we have the decomposition $\Z_{p}^{\times} \backsimeq (\Z / p\Z)^{\times} \times 1 + p\Z_{p} $ (as $p \geq 5$), we can consider the characters restricted to $(\Z / p\Z)^{\times}$. We will only consider the characters that are trivial on the $p-1$st roots of unity. We denote this space by $\mathcal{B}$ (this weight space can be given a rigid analytic structure, but we will not need this). The weight space can be identified with the unit disk $\mathcal{W}$ inside $\C_{p}$, via $\kappa \mapsto \kappa(p+1) -1$. An integral weight $k \in \Z$ will be identified with the character $x \mapsto x^{k}$. For a given weight $\kappa$, we have the Eisenstein series of weight $\kappa$ with $q$-expansion given by
\[
E_{\kappa}^{*}= 1 + \frac{2}{\zeta^{*}(\kappa)}\sum_{n=1}{\infty}
\left( 
\sum_{\substack{d|n\\ p \nmid n}} \kappa(d)d^{-1} 
\right) q^{n} 
\]
(note that we remove the Euler factor at $p$). Here $\zeta^{*}(\kappa)$ is the $p$-adic zeta function.  It is known (see \cite{C}) that $E_{\kappa}^{*}/V(E_{\kappa}^{*})$ is overconvergent, where $V$ is the operator acting on the $q$-expansion by $q \mapsto q^{p}$. From \cite[Theorem B]{OP}, we have the following result.

\begin{thm}\label{thm:main_B} (a) There are modular forms $b_{ij}\in B_i(\Z_p)$ for each $i,j \in \Z_{\ge 0}$ such that the following holds. If $\kappa\in \B \backslash \{ 1\}$ then the Katz expansion of the modular function $\frac{\Es_\k}{V(\Es_\k)}$ is
$$
\frac{\Es_\k}{V(\Es_\k)} = \sum_{i=0}^{\infty} \frac{\b_i(w(\k))}{E_{p-1}^{i}}
$$
where
$$
\b_i(w(\k)) := \sum_{j=0}^{\infty} b_{ij} w(\k)^j
$$
for each $i$.

\noindent (b) There is a constant $c_p$ with $0< c_p <1$ such that for the modular forms $b_{ij}$ in part (a) we have
$$
v_p(b_{ij}) \ge c_p i - j
$$
for all $i,j$.

In fact, we can take 

\[
c_{p} = \frac{2}{3} \left(  1 - \frac{p}{(p-1)^{2}} \right)\frac{1}{p+1}.
\]
\end{thm}
We wish to compute $\nu(b_{i,j})$ and to compare it with the values ensured by the lower bound in the above theorem. The main idea for an algorithm to compute these valuations is using the existence of a "formal Katz expansion", as in statement (a) of Theorem \ref{thm:main_B} and to compute the Katz expansion for enough classical weights to deduce information about this formal Katz expansion. In the next section we will describe two algorithms, the first one will compute the Katz expansion of any weight $0$ overconvergent modular form, and the second algorithm will return (modulo some technicalities) the valuations $\nu(b_{i,j})$.

\section{The Algorithms}
Recall that we have an isomorphism $\psi : \Z_{p}[[q]] \to \prod_{i\geq 0} B_{i}(\Z_{p})$, the inverse of the map $\phi$ introduced in the previous section. Note that this map attaches to a power series (in particular to an overconvergent modular form of weight $0$ and with coefficients in $\Z_{p}$) its Katz expansion. Our first algorithm will have as its goal to compute this Katz expansion, with a modular form (or a power series) as its input. We have to be  careful with the precisions we choose for this. Fix an integer $n \geq 0$ and set $N := d_{n(p-1)}$. Then, for $m>n$, if $g \in B_{m}(\Z_{p})$ we have that $g \equiv 0 \mod{q^{N}}$, so that $\phi$ descends to a map:
$
\phi_{n} : \prod_{i=0}^{n} B_{i}(\Z_{p}) \to \Z_{p}[[q]]/(q^{N}).
$
We have the following:
\begin{lem}\label{alg1proof}
The map \begin{equation}\label{partkatz}
\phi_{n} : \prod_{i=0}^{n} B_{i}(\Z_{p}) \to \Z_{p}[[q]]/(q^{N}), \qquad(b_{i})_{i=0}^{n} \mapsto \sum_{i=0}^{n}  \frac{b_{i}}{E_{p-1}^{i}},
\end{equation} is an isomorphism.
\end{lem}
\begin{proof}
If we are given an element $f \in \Z_{p}[[q]]/(q^{N}) $ and we want to find an inverse, then we have to solve the following matrix system 
\[
Mx = B,
\]
where $M$ is the $N \times N$ matrix whose $j$th column consists of the coefficients of the $q$-expansion of $g_{j}/E_{p-1}^{i_{j}}$, and where $B$ is the column vector consisting of the coefficients of $f$. As $M$ is lower triangular, with $1$'s on the diagonal, the system will have a unique solution. The preimage of $f$ is then given by $(b_{i})_{i=0}^{n} $, where $b_{i} = x_{d_{(i-1)(p-1)}} g_{d_{(i-1)(p-1)}} + \ldots + x_{d_{i(p-1)-1}} g_{d_{i(p-1)-1}}$, which is uniquely determined and hence $\phi_{n}$ is an isomorphism.
\end{proof}
 We denote the inverse of the map given in \eqref{partkatz} by
\[
\psi_{n} : \Z_{p} [[q]] / (q^{N} )\to \prod_{i=0}^{n} B_{i}(\Z_{p}).
\]
The following algorithm computes this map to any given $p$-adic precision, and hence can be seen as computing the partial Katz expansion of a given overconvergent modular form. 
\begin{alg}
Given a prime $p \geq 5$, positive integers $n$ and $C$, and a power series $f$ in $\Z[[q]]/(q^{N},p^{C})$, where $N = d_{n(p-1)}$, this algorithm returns $\psi_{n}(f)$, as an $(n+1)$-tuple, with $p$-adic precision $C$.
\end{alg}
\begin{enumerate}
\item \emph{Dimension of $\mathcal{B}_{i} (\Z_{p})$}: Compute the values $i_{j}$ for the values $j=0, \ldots, N-1$.   \\
\item \emph{Basis of $\mathcal{B}_{i} (\Z_{p})$} Compute the $q$-expansions of the forms $g_{j} \in \mathcal{B}_{i_{j}}$ for $j=0, \ldots, N-1$ up till $q^{N-1}$ with coefficients in $\Z$, using Equation \eqref{gijdef}, after having computed the $q$-expansions of $E_{4}$, $E_{6}$ and $\Delta$ up to $q^{N}$. We normalize $E_{4}$ and $E_{6}$ to have constant coefficient $1$.  \\
\item \emph{Coefficient matrices} Create the $N\times N$ matrix $M$ which has as $j$th column the coefficients of the $q$-expansion $g_{j}E_{p-1}^{-i_{j}}$ up till $q^{N-1}$. This matrix will be lower triangular with $1$s on the diagonal. \\
\item \emph{Katz expansion} Create the column vector $B$, containing the coefficients $a_0(f),\ldots, a_{N-1}(f)$. Solve the equation $Mx = B$ for $x$ over $\Z_{p}/(p^{C})$. Let the solution be given by $x = (x_{0}, \ldots, x_{N-1})^{T}$. Return the tuple $(f_{0},\ldots , f_{n})$, where $f_{i} = x_{d_{(i-1)(p-1)}} g_{d_{(i-1)(p-1)}} + \ldots + x_{d_{i(p-1)-1}} g_{d_{i(p-1)-1}}$.
\end{enumerate}
The correctness of the algorithm is a consequence of Lemma \ref{alg1proof}. Note that steps (1)-(3) only depend on $p$ and $n$, and thus, if working with a fixed prime and precision, the matrix $M$ and the pairs $(i,j)$ found in step (1) can be saved and only step (4) has to be executed.
\\
For the second algorithm, we investigate the "formal Katz expansion" of the family $\Es_{\kappa}/V(\Es_{\kappa})$ as in Theorem \ref{thm:main_B}. The theorem asserts the existence of modular forms $b_{ij}\in B_i(\Z_p)$ for each $i,j \in \Z_{\ge 0}$. The following algorithm allows us to compute $\nu(b_{i,j})$, if the number of weights given as the input is enough, and otherwise it will return "inconclusive", and one will have to adapt the input.

\begin{alg}\label{alg2}
Given a prime $p \geq 5$, a nonnegative integer $r$ and a list $L$ of integral weights $\kappa_{1}, \ldots \kappa_{\lambda}$, (so $\lambda$ = $\#L$)  this algorithm will output the values $\nu(b_{r,j})$, for $0 \leq j \leq r$, if this can be determined exactly, and otherwise it will return "inconclusive" if the choice of weights does not allow us to conclude.
\end{alg}

\begin{enumerate}
\item \emph{Construct the Eisenstein series} First construct the Eisenstein series $E_{\kappa_{i}}^{*}$ for the weights in the given list, and then construct $\Es_{\kappa_{i}}/V(\Es_{\kappa_{i}})$, as elements of $(\Z/p^{\lambda}\Z)[[q]]/(q^{N})$, where $N := d_{r(p-1)}$.\\
\item \emph{Katz expansions} Use Algorithm 1 to compute for all $i =1 ,\ldots, \lambda$ the Katz expansions of $\Es_{\kappa_{i}}/V(\Es_{\kappa_{i}})$ up to the $r$th term, say $\beta_{r}^{(i)}$. This means that we need a precision of $n=r$ and $C = \lambda$ in Algorithm 1.\\
\item \emph{Construct the Vandermonde matrix} Construct the Vandermonde matrix 

\[ V :=
\begin{bmatrix}
1 & w_{1} & \ldots & w_{1}^{\lambda-1} \\
\vdots & & \ddots  & \vdots \\
1 & w_{\lambda} & \ldots & w_{\lambda}^{\lambda-1} \\
\end{bmatrix}, 
\]
over $\Z / p^{\lambda}\Z$ and where $w_{i} = (p+1)^{\kappa_{i}} - 1$. Compute a set of generators for the (left) kernel, which is a subgroup of $(\Z/p^{\lambda}\Z)^{\lambda}$, say $\mathcal{V}$. One can compute a set of generators for the kernel by, for instance, computing the Smith normal form of $V$. Then, for every $1 \leq i \leq \lambda$ compute $\gamma_{i} := \text{min} \{ \nu(v_{i})) | v \in \mathcal{V} \}$, where $v_{i}$ denotes the $i$th component of the vector $v$. \\
\item \emph{Solve linear systems} Define $S := \lceil r(p-1)/12 \rceil$. For $0 \leq i \leq S$, compute the column vector $\theta_{i}$, which has as $l$th entry the $i$th coefficient of $\beta_{r}^{(l)}$, over $\Z / p^{\lambda}\Z$, for $0\leq l \leq \lambda$. Compute the $S+1$ solutions $x_{i}$ of the matrix equations $Vx_{i}=\theta_{i}$, over $\Z / (p^{\lambda})$ (up to an element in the kernel of $V$). \\
\item \emph{Find the minimum valuation} For all $1 \leq j \leq r$ compute for $0 \leq i \leq S$ the minimum of the  values $\nu((x_{i})_{j})$, say $\alpha_{r,j}$, and return the list $[\alpha'_{r,1}, \ldots , \alpha'_{r,r}]$, where $\alpha'_{r,j}=\alpha_{r,j}$ if $ \alpha_{r,j} < \gamma_{j}$, and "$\alpha'_{r,j}$ is inconclusive" if $\alpha_{r,j} \geq \gamma_{j}$. \\
\end{enumerate}
To prove the correctness of Algorithm 2, we rely on Theorem \ref{thm:main_B}, in particular on the existence of the forms $b_{i,j}$. Hence, if the Katz expansion of $\Es_{\kappa_{i}} / V(\Es_{\kappa_{i}})$ is given by $\sum_{j\geq0} \beta_{j}^{(i)}/ E_{p-1}^{j}$, we get the following equation
\[
\beta_{j}^{(i)} = \sum_{l=0}^{\infty}b_{j,l}w_{i}^{l}.
\]
Now, $\nu(w_{i}) \geq 1$ and $\nu(b_{i,j}) \geq 0$, so reducing modulo $p^{\lambda}$ then gives us

\[
\beta_{j}^{(i)} \equiv \sum_{l=0}^{\lambda -1} b_{j,l}w_{i}^{l}   \mod{p^{\lambda}}.
\]
As we get this for every weight $\k_{1},\ldots, \k_{\lambda}$ we obtain, for any $\mu \in \N$,  the following matrix equation over $\Z/p^{\lambda}\Z$:

\[
\begin{bmatrix}
1 & w_{1} & \ldots & w_{1}^{\lambda-1} \\
\vdots & & \ddots  & \vdots \\
1 & w_{\lambda} & \ldots & w_{\lambda}^{\lambda-1} \\
\end{bmatrix} 
\begin{bmatrix}
a_{\mu}(b_{r,0}) \\
\vdots \\
a_{\mu}(b_{r,\lambda-1}) \\
\end{bmatrix}
=
\begin{bmatrix}
a_{\mu}(\beta_{r}^{(1)}) \\
\vdots \\
a_{\mu}(\beta_{r}^{(\lambda)})\\
\end{bmatrix},
\]
where $a_{\mu}$ denotes the $\mu$-th Fourier coefficient in the $q$-expansion. Note that the right hand side is known, as these are the coefficients appearing in the Katz expansion, which we can compute using Algorithm 1. We can solve this linear system, giving us a solution for $a_{\mu}(b_{r,0}), \ldots, a_{\mu}(b_{r,\lambda})$. Note that we know such a solution exists, as the forms $b_{i,j}$ exist, but this matrix equation might not have a unique solution, as $V$ (the Vandermonde matrix) is in general not invertible over $\Z/p^{\lambda}\Z$. However, we can find a set of generators for the kernel of this matrix. If we let $\gamma_{i}$ denote the minimum of the valuations of the $i$th entry of this set of generator, then we do know that if the valuation of the solution we find, say $\nu(a_{\mu}(b_{r,i}) < \gamma_{i}$, then this is the same valuation for any other solution. To conclude that $\nu(b_{r,i})$ is the minimum of $\nu(a_{\mu}(b_{r,i}))$, we apply the following lemma:
\begin{lem}
If $f \in M_{n}(\Z)$ and $\nu(a_{i}(f)) \geq b$ for $i=0, \ldots, \lceil n/12 \rceil$, then $\nu(f) \geq b$.
\end{lem}
\begin{proof}
We will use Sturm's theorem, which says that if $\mathfrak{m}$ is a prime ideal of the ring of integers $\mathcal{O}$ of a number field $K$, $\Gamma$ a congruence subgroup of $\SL_{2}(\Z)$ of index $m$ and $f \in M_{k}(\Gamma,\mathcal{O})$ such that
\[
 \text{ord}_{q}(f \text{ mod }{\mathfrak{m}}) > km/12 ,
\]
then $f \equiv 0 (\text{ mod }{\mathfrak{m}})  $ \cite{S}. We will apply Sturm's theorem with $K =\Q$, $\mathcal{O}=\Z$ and $\mathfrak{m} = (p)$ and use induction to prove the lemma. The case that $b=1$ immediately follows from Sturm's theorem. Then, for the general case, the induction hypothesis implies that $p^{b-1}f \in M_{k}(\mathcal{O})$ and we apply Sturm's theorem to $p^{b-1}f $, which implies that $p^{b-1}f \equiv 0 \mod(p)$, and hence $\nu(f) \geq b$.
\end{proof}
Hence, for a fixed $0 \leq l \leq r$, if we know $\nu(a_{\mu}(b_{r,l}))$ for $0 \leq \mu \leq \lceil r(p-1)/12 \rceil$, then we know $\nu(b_{r,l})$. To apply this in the algorithm, we need $\beta_{r}^{(i)}$, the $r$th term of the Katz expansion of  $\Es_{\kappa_{i}}/V(\Es_{\kappa_{i}})$, up till a precision of $S := \lceil (p-1)\cdot r/12 \rceil$. However, we remark that $d_{r(p-1)} \geq  \lceil r(p-1)/12 \rceil$ and hence Algorithm 1 returns the Katz expansions with sufficient precision.\\ 
As noted before, the algorithm only finds $\nu (b_{r,j})$ if it is less than the $\gamma_{j}$, since $b_{r,j}$ is only found up to an element in the kernel of $V$. The following lemma shows that we can make $\gamma_{j}$ arbitrarily high by increasing the number of weights.

\begin{lem}\label{kernellemma}
For any $i \in \Z_{\geq 1}$ and $\gamma \in \Z_{\geq0}$, there exist $n \in \Z$ and $w_{1} ,\ldots , w_{n} \in \Z_{p}$ with $\nu_{p}(w_{j})=1$ for all $1 \le j \le n$,  such that for any $\alpha  \in \ker{ (\overline{V})}$, with 
\[
\overline{V} = 
\begin{bmatrix}
1 & w_{1} & \ldots & w_{1}^{n-1} \\
\vdots & & \ddots  & \vdots \\
1 & w_{n} & \ldots & w_{n}^{n-1} \\
\end{bmatrix}  
\in M_{n \times n} ( \Zmodpm{p^n}),
\]
we have $\nu_{p} (\alpha_{i}) \geq \gamma$, (where $\a_{i}$ is the $i$th component of $\a$).
\end{lem}
\begin{proof}
Given $ w_{1}, \ldots , w_{n} \in  \Z_{p}$ with $\nu_{p}(w_{j})=1$ for all $1 \le j \le n$, we denote 
\[
V := \begin{bmatrix}
1 & w_{1} & \ldots & w_{1}^{n-1} \\
\vdots & & \ddots  & \vdots \\
1 & w_{n} & \ldots & w_{n}^{n-1} \\
\end{bmatrix}  
\in M_{n \times n} ( \Q_{p}),
\]
and we denote by $\overline{V}$ the matrix as in the statement of the lemma (i.e. the matrix $V$ reduced modulo $p^{n}$). Note that since the $w_{j}$ are distinct, $V$ will be invertible over $\Q_{p}$. In particular, if we have an element $\alpha \in \ker{ (\overline{V})}$, we can lift this to a vector $\tilde{\a} \in \Z_{p} ^{n}$, such that $\nu_{p} (\a_{i}) = \nu_{p} (\tilde{a}_{i})$, and thus
\[
V\tilde{\a} = \begin{bmatrix}
p^{n} b_{1} \\
\vdots \\
p^{n} b_{n} \\
\end{bmatrix},
\]
where the $b_{1}, \ldots, b_{n} \in \Z_{p}$. This implies

\[
\tilde{\alpha} = V^{-1} \begin{bmatrix}
p^{n} b_{1} \\
\vdots \\
p^{n} b_{n} \\
\end{bmatrix},
\]
and hence
\[
\nu_{p} (\a_{i}) = \nu_{p} (\tilde{a}_{i}) \geq n + \min_{1\leq j \leq n} \nu_{p}((V^{-1})_{i,j}).
\]
Hence, it remains to bound the valuations appearing in the inverse of the Vandermonde matrix $V$. From \cite[Lemma 3.1]{OP} we know that the coefficient of $V^{-1}$ at position $(i,j)$ is given by
\[ (-1)^{n-i} \cdot \frac{s_{n-i}(w_1,\dots,\hat{w_j},\dots,w_{n})}{\prod_{0 \le \ell \le n, \ell \neq j} (w_j - w_\ell)},\]
where $s_d(\dots)$ is the elementary symmetric polynomial of degree $d$ in $n-1$ variables (the hat meaning that this variable is omitted). As all the $w_{j}$ have valuation $1$, we find that $\nu_{p} ( s_{n-i}(w_1,\dots,\hat{w_j},\dots,w_{n}) ) \geq n-i$. As for the denominator, we have
\begin{equation}\label{vander}
\max_{1 \le j \le n} \nu_{p} \left( \prod_{ \substack{ 1 \leq \ell \leq n \\ \ell \neq j } } (w_{l} - w_{j}) \right) = \max_{1 \le j \le n} \nu_{p} \left( p^{n} \prod_{ \substack{ 1 \leq \ell \leq n \\ \ell \neq j } } \left( \frac{w_{l}}{p} - \frac{w_{j}}{p} \right) \right) \geq n - 1 + f(n),
\end{equation}
where 
\[
f(n) := \sum_{i=1}^\infty \left\lfloor \frac{n-1}{(p-1)p^{i-1}} \right\rfloor.
\]
The last inequality in \eqref{vander} is \cite[Proposition 3.2]{OP} and becomes an equality if the $w_{j}$ are chosen correctly (compare with \cite[Proof 3.3]{OP}). We note that 
\[
f(n) \leq \frac{n-1}{p-1} \sum_{i=0}^{\infty} \frac{1}{p^{i}} = (n-1)\frac{p}{(p-1)^{2}},
\] 
and thus, assuming the $w_{j}$ are chosen such that we have equality in \eqref{vander}, putting everything together we find
\[
\min_{1\leq j \leq n} \nu_{p}((V^{-1})_{i,j}) \geq n - i - \left( n - 1 + (n-1)\frac{p}{(p-1)^{2}} \right) = 1  - i - (n-1)\frac{p}{(p-1)^{2}}.
\]
We conclude that
\[
\nu_{p} (\a_{i}) \geq n + 1 - i - (n-1)\frac{p}{(p-1)^{2}} = n \cdot \left( 1 - \frac{p}{(p-1)^{2}} \right) - i + 1+  \frac{p}{(p-1)^{2}},
\]
but, as $\frac{p}{(p-1)^{2}} < 1$ for $p \geq 5$, the right hand side can be made arbitrarily high by increasing $n$.

\end{proof}
Note that the proof requires us to choose the weights correctly. To show that we can indeed do this, we have the following result.

\begin{lem}\label{valuationlem}
If $a,b \in \N$ and $p\geq 3$ prime, then we have 
\[
\nu_{p}\left( \frac{(1+p)^{a}-(1+p)^{b}}{p} \right) = \nu_{p}(a-b). 
\] 
\end{lem}
\begin{proof}
Assume without loss of generality that $a \geq b$, then
\[
 \frac{(1+p)^{a}-(1+p)^{b}}{p} =  \frac{(1+p)^{a}(1-(1+p)^{b-a})}{p}
\]
and $\nu_{p}((1+p)^{a}) = 0$, so it suffices to show that for any $c \in \N_{\geq 1}$ we have
\[
\nu_{p}\left( (1+p)^{c} - 1 \right) = \nu_{p}(c) + 1.
\]
We have
\begin{equation}\label{binom}
(1+p)^{c} - 1 = cp + {c \choose 2 }p^{2} +  {c \choose 3 }p^{3} + \ldots + p^{c}.
\end{equation}
If $n \geq 2$, then 
\[
\nu_{p} \left( {c \choose n } \right) = \nu_{p} \left( \frac{c}{n}\right) + \nu_{p} \left( {c - 1\choose n - 1} \right) \geq \nu_{p}(c) -\nu(n),
\]
as ${c - 1\choose n - 1} $ is a positive integer. Furthermore, we have $
\nu_{p}(n) \leq n- 2$ and hence 
\[
\nu_{p} \left( {c \choose n } \right) > \nu_{p}(c) + 1 -n.
\]
This shows that if we take the valuation of the right hand side of\eqref{binom} we end up with
\[
\nu_{p}\left( (1+p)^{c} - 1 \right) = \nu_{p}(cp) = \nu_{p}(c) + 1.
\]
\end{proof}
In particular, if we let $S$ be the set containing the first $n$ natural numbers prime to $p$, then we have that 
\[
\max_{x \in S} \nu_{p} \left( \prod_{s \in S, s \neq x} (x - s) \right) = f(n),
\]
as explained in \cite[Lemma 3.1]{OP}. But now Lemma \ref{valuationlem}, shows that if we take the classical weights, $\{k_{s}: x \mapsto x^{s(p-1)} | s \in S\}$ with the corresponding $\{ w_{s} = (p+1)^{s(p-1)}-1 | s \in S \}$, then 
\[
f(n) = \max_{x \in S} \nu_{p} \left( \prod_{s \in S, s \neq x} (x - s) \right) = \max_{x \in S} \nu_{p} \left( \prod_{s \in S, s \neq x } \left( \frac{w_{x}}{p} - \frac{w_{s}}{p} \right) \right) \,
\]
where the second equality follows from Lemma \ref{valuationlem}, as it implies that for all $x,s \in S$ we have the equality $\nu_{p} (x - s ) = \nu_{p}((w_{x} - w_{s})/p)$. Thus Lemma \ref{kernellemma} applies to this choice of weights. In particular, if Algorithm 2 returns "inconclusive", one can increase the number of weights given as the input as above, which will increase the values of the $\gamma_{j}$. Hence, as long as the $b_{i,j}$ are non-zero, there will be a number of weights such that $\gamma_{j} > b_{i,j}$ and thus we can determine $\nu(b_{i,j})$ exactly. However, if for some $i$ and $j$ we have that $b_{i,j} = 0$, then we cannot use this algorithm to determine $v(b_{i,j})$ (as it will be infinite). Computations so far seem to suggest that if $i,j \neq 0$ and $\mathcal{B}_{i} \neq \emptyset$, then $b_{i,j} \neq 0$, but a proof does not seem available at the time, nor are we sure to even expect that this is the case. However, if $j = 0$ we do have the following result.
\begin{prop}\label{jzero}
Let the $b_{i,j}$ be as in Theorem A. Then $b_{0,0} = 1$ and $b_{i,0} = 0$ for $i>0$.
\end{prop}
To prove this, we need the following lemma.
\begin{lem}\label{lemminverse}
Let $R$ be a commutative ring and $f \in 1 + xyR[[x,y]]$. Then $f$ is invertible and $f^{-1} \in 1 + xyR[[x,y]]$.
\end{lem}
\begin{proof}
If we consider $f$ as a power series in the variable $y$ and with coefficients power series in $x$, say $f = 1 + a_1(x)y + a_2(x)y^{2}+ \ldots$, then the inverse is given by $f^{-1} = 1 + b_1(x)y + b_2(x)y^{2}+ \ldots $, where $b_{k}(x) = -\sum_{i=1}^{k} a_{i}b_{k-i}$, where we set $b_{0} = 1$. Induction then shows that $b_{i}(x) \in xR[[x]]$ and hence $f^{-1} \in 1 + xyR[[x,y]]$.
\end{proof}

Now we can prove Proposition \ref{jzero}.
\begin{proof}

We first note that the isomorphism $\phi$ descends to an isomorphism 
\[
\tilde{\phi} : \prod_{i\ge 0} B_i(\Z/p^{n}\Z) \to \left( \Z / p^{n}\Z \right)[[q]],
\]
for any $n \in \N$, which follows from the same argument that $\phi$ is an isomorphism. Now, let $\kappa$ be an integral weight such that $w := w_{\kappa}$ satisfies $\nu(w) \geq n$ (one can pick for example the integral weight corresponding to $k = (p-1)p^{n+1}$). Then, see \cite{BuzCal}, we know that $\Es_{\kappa} \in 1 +wq\Z_{p}[[w,q]]$ and $V(\Es_{\kappa}) \in 1 +wq\Z_{p}[[w,q]]$ and Lemma \ref{lemminverse} then implies that $\frac{\Es_\k}{V(\Es_\k)}  \in 1 +wq\Z_{p}[[w,q]] $.
Looking at the Katz expansion 
\[
\frac{\Es_\k}{V(\Es_\k)} = \sum_{i=0}^{\infty} \frac{\sum_{j=0}^{\infty} b_{ij}w^{j}}{E_{p-1}^{i}},
\]
and reducing modulo $p^{n}$ we get 
\[
1 \equiv \sum_{i=0}^{\infty}  \frac{b_{i,0}}{E_{p-1}^{i}} \mod{p^{n}},
\]
as $\nu(w) \geq n$. This is a congruence of power series. Note that for a fixed $N$, there are only finitely many $i$ such that $b_{i,0}$ has a non-zero coefficient of $q^{N}$ and hence the infinite sum makes sense. As $\phi_{n}$ is an isomorphism, we get that $b_{0,0} \equiv  1 \mod {p^{n}} $ and $b_{i,0} \equiv 0 \mod{p^{n}}$ for $i>0$, since $b_{0,0} \in \Z_{p}$ and $b_{i,0}$ has no constant term for $i>0$. But this holds for any $n$ and we conclude.
\end{proof}

\section{Observations}
In this section we will present data, which is obtained using Algorithm 2. As already stated in the introduction, we define  $\delta_{p} \in \R$ as follows:
\[
\delta_{p} := \text{inf} \left \{ \frac{ \nu_{p}(b_{ij}) + j}{i} \middle| i \in \Z_{>0}, j \in \Z_{\geq 0} \right \}.
\]
Our main objective will be to provide an upper bound for the value $\delta_{p}$ obtained from computations and use this to formulate a precise conjecture on a lower bound for $\delta_{p}$. As we can only compute finitely many values for $\nu(b_{i,j})$, we can only obtain an upper bound. From Theorem \ref{thm:main_B} we know that there exists a constant $c_{p}$ such that $\delta_{p} \geq c_{p}$. The main intent for Algorithm 2 is to see whether we expect $c_{p} = \delta_{p}$ or whether we expect $c_{p}$ to be strictly smaller than $\delta_{p}$. To be more precise, assume we have a set of tuples $(i,j) \in \N_{\geq 1} \times \N_{\geq 0}$, say $S$, for which we have computed $\nu(b_{i,j})$ (so $S$ is necessarily finite), then an upper bound for $\delta_{p}$ can be given by computing
\begin{equation}\label{mineq}
d'_{p} := \min_{(i,j) \in S}\left\{  \frac{\nu(b_{i,j})+j}{i} \right\}.
\end{equation}

\begin{remark}
Note that for a fixed $i$ we do not need to know the values $\nu(b_{i,j})$ for all $j$. More precisely, if we have already found an upper bound for $\delta_{p}$, say $d'$, then we only need to know the values  $\nu(b_{i,j})$ for $j \leq d'i$. Indeed, as $\nu(b_{i,j}) \geq 0$ (since the $b_{i,j}$ have coefficients in $\Z_{p}$), we have that if  $j \geq d'i$, then $\frac{ \nu(b_{i,j})+j}{i} \geq d'$ and hence computing $\nu(b_{i,j})$ for larger values of $j$ will have no impact on an upper bound of $d_{p}$.
\end{remark}

We ran Algorithm 2 for the primes $p \in \{5,7,11,13,17,37 \}$ for different ranges of $i$. Once we find a new upper bound $d'_{p}$, we input enough weights so that we get exact values of $\nu(b_{i,j})$ for $j=1,\ldots, \lceil d'_{p}i \rceil $, as the values $\nu(b_{ij})$ for $j \geq d_{p}'i$ will not influence the minimum as in \eqref{mineq}, as explained in the remark above.
In the following table we present the values for the upper bounds we have found. We include till what value of $i$ we computed $\nu(b_{i,j})$ and we also include the first value of $i$ for which the value $d_{p}'$ is attained.

\begin{center}
\begin{tabular}{ |c|c|c|c| } 
 \hline
 prime & $i\leq $ & $d'_{p}$ & attained at $i=$ \\  \hline
 5 & 599 & 2/15 & 30\\ \hline
 7 & 502 & 3/28 & 56\\  \hline
 11 &312&  5/66 & 132\\ \hline
 13 & 288 & 6/91 & 182 \\ \hline
 17 & 248 & 1/18 & 18 \\ \hline
 37 & 130 & 1/38 & 38 \\ \hline
\end{tabular}
\end{center}
\medskip
For a more visual representation, we also include the following plot for $p=11$.

\begin{figure}[h]
\begin{center}

\begin{subfigure}{.4\textwidth}
  \centering
  \includegraphics[scale = 0.6]{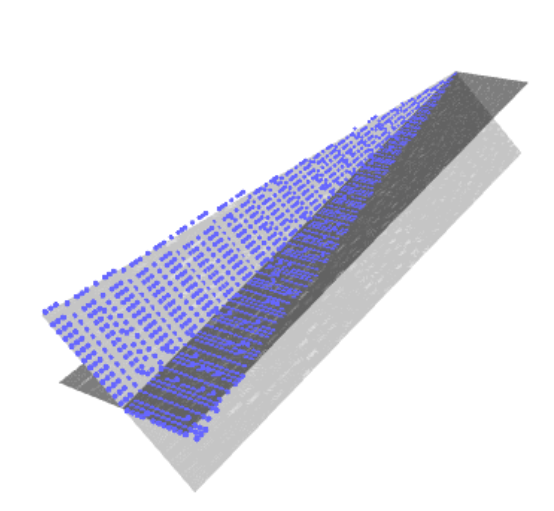}
  \label{fig:sub1}
\end{subfigure}%
\begin{subfigure}{.6\textwidth}
  \centering
  \includegraphics[scale = 0.6]{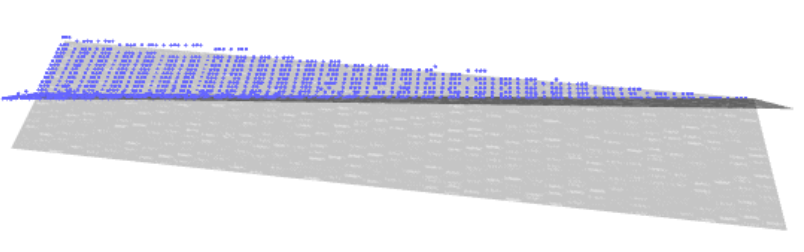}
  \label{fig:sub2}
\end{subfigure}
\caption{The values $(i,j, \nu(b_{i,j}))$ for $1 \leq i \leq 172$ and $1 \leq j \leq 5/66 \cdot i$ and $p=11$. The black plane is given by $z=0$ and the grey plane is given by $5/66 \cdot i - j =z$. }
\label{fig:test}

\end{center}
\end{figure}
Note that there are some points $(i,j)$ missing for which Algorithm \ref{alg2} returned inconclusive. In particular, this happens for certain values for which $\mathcal{B}_{i} = \emptyset$, e.g. for $i= 7$, and hence $b_{i,j}=0$ for these $i$.\\

Based on this data we formulated Conjecture \ref{conjecture}.
We would like to make a few remarks about this. First, we note that we only found an upper bound in agreement with our conjectured value of $d_{p}$ in the cases $p=5,7,11,13$. In the cases $p=17$ and $p=37$ we find a value $1/(p+1)$ as an upper bound instead, strictly larger than our conjectured value for $d_{p}$. However, due to the nature of how we compute an upper bound, see \eqref{mineq}, we can only get a value $i$ in the denominator if we have computed $\nu(b_{i,j})$ for a multiple of $i$. In particular, for $p=17$, we expect to find $d_{p}' = 8/153$, which means that we can only find this value if we compute $\nu(b_{i,j})$ for $i$ a multiple of $153$. Similarly, for $p=13$, we need to compute $\nu(b_{i,j})$ where $i$ is a multiple of $703$. As the computation time increases as we increase $i$, we have not been able to compute  $\nu(b_{i,j})$ for these values of $i$.\\

Secondly, assuming that $p=5$ or $p=7$, and $k\in \N$ divisible by $p-1$, then we know that 
\[
\frac{ \Es_{k} } { V(\Es_{k}) } \in M_{0} \left( \Z_{p}, \geq \frac{p-1}{p(p+1)} \right) ,
\]
see Proposition 4.2 in \cite{OP}. In particular, we see that for these primes, our conjectured value for $d_{p}$ precisely agrees with the overconvergence rate of Eisenstein series with classical weights. However, the proof of this is highly specific for $p=5,7$, and for larger primes than this, only strictly higher overconvergent rates are proven. Furthermore, we are not sure if it is possible to theoretically prove Conjecture \ref{conjecture} using the overconvergent rates for the Eisenstein series with only classical weights.

\bibliography{ocalgorithmsbib} 
\end{document}